\documentclass[11pt]{article}
\usepackage{amsmath,amsthm,amsfonts,amssymb,bm,wasysym}
\usepackage{subfigure}
\usepackage{graphicx}
\usepackage[usenames]{color}
\usepackage{verbatim}
\usepackage{epsfig}

\parskip \medskipamount
\newtheorem{theorem}{Theorem}[section]

\numberwithin{equation}{section}
\newtheorem{lemma}[theorem]{Lemma}

\newtheorem{remark}[theorem]{Remark}

\newtheorem{claim}[theorem]{Claim}

\numberwithin{equation}{section}

\def\N{\mathbb{N}}
\def\Z{\mathbb{Z}}

\def\EE{\mathcal{E}}

\def\T{\mathcal{T}}

\renewcommand{\phi}{\varphi}
\renewcommand{\epsilon}{\varepsilon}

\newcommand{\1}{{\text{\Large $\mathfrak 1$}}}

\newcommand{\til}{\widetilde}

\newcommand{\tmix}{t_{\mathrm{mix}}}

\newcommand{\trel}{t_{\mathrm{rel}}}

\newcommand{\pr}[1]{\mathbb{P}\!\left(#1\right)}
\newcommand{\E}[1]{\mathbb{E}\!\left[#1\right]}
\newcommand{\estart}[2]{\mathbb{E}_{#2}\!\left[#1\right]}
\newcommand{\prstart}[2]{\mathbb{P}_{#2}\!\left(#1\right)}
\newcommand{\prcond}[3]{\mathbb{P}_{#3}\!\left(#1\;\middle\vert\;#2\right)}
\newcommand{\econd}[2]{\mathbb{E}\!\left[#1\;\middle\vert\;#2\right]}
\newcommand{\econds}[3]{\mathbb{E}_{#3}\!\left[#1\;\middle\vert\;#2\right]}

\newcommand{\var}[1]{\operatorname{Var}\!\left(#1\right)}
\newcommand{\vars}[2]{\operatorname{Var}_{#2}\!\left(#1\right)}

\newcommand{\tn}{|\kern-.1em|\kern-0.1em|}

\newcommand{\enk}[1]{\estart{#1}{n_k}}
\newcommand{\vnk}[1]{\operatorname{Var}_{n_k}\!\left(#1\right)}

\newcommand\be{\begin{equation}}
\newcommand\ee{\end{equation}}

\begin{document}

\title{\bf Total variation cutoff in a tree}

\author{
Yuval Peres\thanks{Microsoft Research, Redmond, Washington, USA; peres@microsoft.com} 
\and Perla Sousi\thanks{University of Cambridge, Cambridge, UK;   p.sousi@statslab.cam.ac.uk}
}
\maketitle
\thispagestyle{empty}

\begin{abstract}
We construct a family of trees on which a lazy simple random walk exhibits total variation cutoff. 
The main idea behind the construction is that hitting times of large sets should be concentrated around their means. For this sequence of trees we compute the mixing time, the relaxation time and the cutoff window. 
\newline
MSC 2010 \emph{subject classifications.}
Primary   60J10.   
\end{abstract}

\section{Introduction}

Let $X$ be an irreducible aperiodic Markov chain on a finite state space with stationary distribution~$\pi$ and transition matrix~$P$. The lazy version of~$X$ is a Markov chain with transition matrix~$(P+I)/2$. Let $\epsilon>0$. The $\epsilon$-total variation mixing time is defined to be
\[
\tmix(\epsilon) = \min\{ t\geq 0: \max_x\|P^t(x,\cdot) - \pi\|\leq \epsilon\},
\]
where $\|\mu - \nu\| = \sup_{A}|\mu(A) - \nu(A)|$ is the total variation distance between the measures $\mu$ and $\nu$.

We say that a sequence of chains $X^n$ exhibits total variation cutoff if for all $0<\epsilon <1$ 
\[
\lim_{n\to \infty} \frac{\tmix^{(n)}(\epsilon)}{\tmix^{(n)}(1-\epsilon)} = 1.
\]
We say that a sequence $w_n$ is a cutoff window for a family of chains $X^n$ if $w_n = o(\tmix(1/4))$ and for all $\epsilon>0$ there exists a positive constant $c_\epsilon$ such that for all $n$
\[
\tmix(\epsilon) - \tmix(1-\epsilon) \leq c_\epsilon w_n.
\]
Loosely speaking cutoff occurs when over a negligible period of time the total variation distance from stationarity drops abruptly from near $1$ to near $0$. It is standard that if $\trel$ and $\tmix$ are of the same order, then there is no cutoff (see for instance~\cite[Proposition~18.4]{LevPerWil}). From that it follows that a lazy simple random walk on the interval $[0,n]$ or a lazy simple random walk on a finite binary tree on $n$ vertices do not exhibit cutoff, since in both cases $\trel \asymp \tmix$. 

Although the above two extreme types of trees do not exhibit cutoff, in this paper we construct a sequence of trees, where a lazy simple random exhibits total variation cutoff. We start by describing the tree and then state the results concerning the mixing and the relaxation time of the lazy simple random walk on it.

Let $n_j=2^{2^j}$ for $j\in \N$. We construct the tree $\T$ of Figure~\ref{fig:tree} by placing a binary tree at the origin consisting of $N=n_k^3$ vertices. Then for all $j\in \{[k/2],\ldots, k\}$ we place a binary tree at distance $n_j$ from the origin consisting of $N/n_j$ vertices.

For each $j$ we call $\T_j$ the binary tree attached at distance $n_j$ and $\T_0$ the binary tree at $0$. We abuse notation and denote by $n_j$ the root of $\T_j$ and by $0$ the root of $\T_0$.

\begin{figure}[h!]
\begin{center}
\includegraphics[scale=0.85]{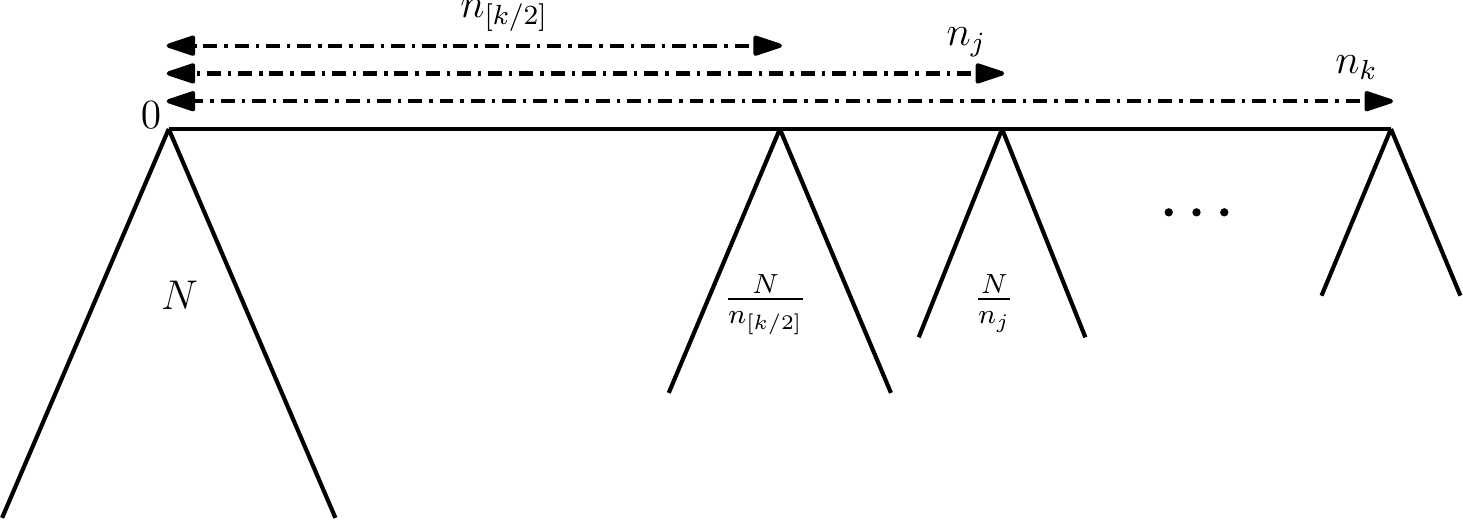}
\caption{\label{fig:tree} The tree $\T$ (not drawn to scale)}
\end{center}
\end{figure}

\begin{theorem}\label{thm:tree}
The lazy simple random walk on the tree $\T$ exhibits total variation cutoff and for all $\epsilon$ 
\[
\tmix(\epsilon) \sim 6Nk. 
\]
Further, the cutoff window is of size $N\sqrt{k}$, i.e.\ for all $0<\epsilon<1$
\[
\tmix(\epsilon) - \tmix(1-\epsilon) \leq c_\epsilon N\sqrt{k},
\]
where $c_\epsilon$ is a positive constant.
\end{theorem}

By Chen and Saloff-Coste~\cite{ChenSaloff1} cutoff also holds for the continuous time random walk on $\T$.

The main ingredient in the proof of Theorem~\ref{thm:tree} is to establish the concentration of the first hitting time of $0$ starting from $n_k$. Once this has been completed, cutoff follows easily. In Section~\ref{sec:concentration}, we prove the concentration result of the hitting time, which then gives a lower bound on the mixing time. Then in Section~\ref{sec:coupling} we describe the coupling that will yield the matching upper bound on the mixing time.

\begin{remark}\rm{
We note that the same idea of showing concentration of hitting times was used in~\cite{DingLubPer} in order to establish cutoff for birth and death chains satisfying $\tmix \text{gap} \to \infty$. Connection of hitting times to cutoff is presented in greater generality in~\cite{CarFraSc}.
}
\end{remark}

It follows from Theorem~\ref{thm:tree} and~\cite[Proposition~18.4]{LevPerWil} that $\trel=o(\tmix(1/4))$. In the next theorem we give the exact order of the relaxation time. We prove it in Section~\ref{sec:relaxation}.

We use the notation $a_k\asymp b_k$ if there exists a constant $C$ such that $C^{-1} b_k \leq a_k \leq C b_k$  for all $k$ and we write $a_k\lesssim b_k$ if there exists a constant $C'$ such that $a_k\leq C'b_k$ for all $k$.

Let $1=\lambda_1\geq \lambda_2\geq \lambda_3,\ldots$ be the eigenvalues of a finite chain. Let 
$\lambda_*=\max_{i\geq 2}|\lambda_i|$ and define the relaxation time $\trel=(1-\lambda_*)^{-1}$.
Note that for a lazy chain $\lambda_*=\lambda_2$.

\begin{theorem}\label{thm:relax}
The relaxation time for the lazy simple random walk on the tree $\T$ satisfies
\[
\trel \asymp N.
\]
\end{theorem}

To the best of our knowledge the tree $\T$ is the first example of a tree for which $\tmix$ is not equivalent to $\trel$. A related problem was studied in~\cite{PeresSousi} and we recall it here.

Suppose that we assign conductances to the edges of a tree in such a way that $c\leq c(e)\leq c'$ for all edges $e$, where $c$ and $c'$ are two positive constants. It is proved in~\cite[Theorem~9.1]{PeresSousi} that the mixing time of the weighted lazy random walk is up to constants the same as the mixing time of the lazy simple random walk on the original tree. Since the relaxation time is given by a variational formula, it is immediate that after assigning bounded conductances $\trel$ is only changed up to multiplicative constants. Hence if in the original tree we have that $\trel$ and $\tmix$ are of the same order, then there is no way of assigning weights to the edges in order to make the weighted random walk exhibit cutoff.

\section{Concentration of the hitting time}\label{sec:concentration}

Let $X$ denote a lazy simple random walk on the tree $\T$. Define for all $x\in \T$
\[
\tau_x = \inf\{s\geq 0: X_s = x\}.
\]

\begin{lemma}\label{lem:concentration}
We have as $k\to \infty$
\[
\enk{\tau_0} = 6Nk+ o\left(N\sqrt{k}\right) \ \text{ and } \ \vnk{\tau_0}\asymp N^2 k.
\]
\end{lemma}

We will prove the above concentration lemma in this section. We start by stating standard results about hitting times and excursions that will be used in the proof of Lemma~\ref{lem:concentration}. We include their proofs for the sake of completeness.

\begin{claim}\label{cl:lazynonlazy}
Let $\tau$ and $\til{\tau}$ denote hitting times of the same state for a discrete time non-lazy and lazy walk respectively. Then
\[
\E{\til{\tau}} = 2\E{\tau} \ \ \text{ and } \ \ \var{\til{\tau}} = 4 \var{\tau} + 2\E{\tau},
\]
assuming that both the lazy and non-lazy walks start from the same vertex.
\end{claim}

\begin{proof}[{\bf Proof}]
It is easy to see that we can write $\til{\tau}= \sum_{i=0}^{\tau-1} \xi_i$, where $(\xi_i)_i$ is an i.i.d.\ sequence of geometric variables with success probability $1/2$. By Wald's identity we get
\[
\E{\til{\tau}} = \E{\tau} \E{\xi_1} = 2\E{\tau}.
\]
Using the independence between $\tau$ and the sequence $(\xi_i)_i$ gives the identity for the variance of $\til{\tau}$.
\end{proof}

\begin{claim}\label{cl:excursion}
Let $T$ be the time spent in an excursion from the root by a simple random walk in a binary tree of size $n$. Then 
\[
\E{T} = \frac{3n-1}{2} \ \ \text{ and } \ \  \E{T^2} \asymp n^2.
\]
\end{claim}

\begin{proof}[{\bf Proof}]
It is standard that $\E{T} = \pi(o)^{-1} = (3n-1)/2$. It is easy to see that starting from any point~$x$ on the tree, the expected hitting time of the root is upper bounded by $cn$ for a constant $c$. Hence by performing independent experiments and using the Markov property we get for a positive constant~$c'$
\[
\pr{T>2kn} \leq e^{-c'k}.
\]
Therefore, we deduce that $\E{T^2} \leq c''n^2$.
\end{proof}

\begin{claim}\label{cl:localtime}
Let $X$ be a simple random walk on the interval $[0,n]$ started from $n$ and~$L_i$ be the number of visits to~$i$ before the first time $X$ hits $0$. Then $L_i$ is a geometric random variable with parameter~$(2i)^{-1}$.
\end{claim}


We are now ready to give the proof of the concentration result.

\begin{proof}[{\bf Proof of Lemma~\ref{lem:concentration}}]

By Claim~\ref{cl:lazynonlazy} it suffices to consider a non-lazy random walk. We write~$\tau_0$ for the first hitting time of~$0$ for a simple random walk on the tree~$\T$. 

%
%
Every time we visit a vertex~$n_j$ for some~$j$ with probability~$1/2$ we make an excursion in the binary tree attached to this vertex. Since we are interested in the time it takes to hit~$0$ we can think of the problem in the following way: we replace a binary tree by a self-loop representing a delay which is the time spent inside the tree in an excursion from the root. It will be helpful to have Figure~\ref{fig:delay-fig} in mind.

\begin{figure}[h!]
\begin{center}
\includegraphics[scale=0.85]{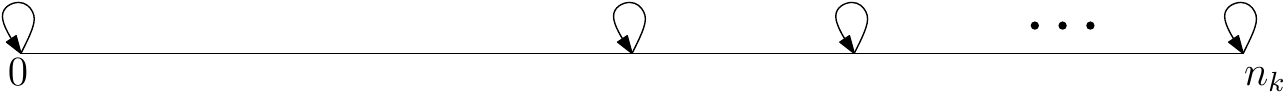}
\caption{\label{fig:delay-fig} Delays represented by self loops}
\end{center}
\end{figure}

Let~$Y$ be a simple random walk on the line~$[0,n_k]$ starting from~$n_k$. Let~$S$ be the time it takes~$Y$ to reach~$0$. 
For~$i=[k/2],\ldots, k$ we let~$L_i$ be the local time at~$n_i$ before the first time~$Y$ hits~$0$, i.e.
\[
L_i = \sum_{\ell=0}^{S} \1(Y_\ell=n_i). 
\]
For every vertex $n_i$ we let $(T_\ell^{(i)})_{\ell\leq L_i}$ be the delays incurred during the $L_i$ visits to $n_i$, i.e.
\[
T_\ell^{(i)} = \sum_{m=1}^{G_{i,\ell}} \xi_{m}^{(i)},
\]
where $(\xi_{m}^{(i)})_m$ is an i.i.d.\ sequence of excursions from $n_i$ in the binary tree rooted at $n_i$ and $G_{i,\ell}$ is an independent geometric random variable of success probability $1/2$. Note that the random variables $T_\ell^{(i)}$ are independent over different $i$ and $\ell$. Having defined these times we can now write
\begin{align}\label{eq:decomp}
\tau_0 = S + \sum_{i=[k/2]}^{k} \sum_{\ell=1}^{L_i} T_\ell^{(i)} = S+D,
\end{align}
where $D =  \sum_{i=[k/2]}^{k} \sum_{\ell=1}^{L_i} T_\ell^{(i)}$. 
From Claims~\ref{cl:localtime} and~\ref{cl:excursion} and the independence between~$L_i$ and $T_\ell^{(i)}$ using the above representation of~$\tau_0$ we immediately get
\[
\enk{\tau_0} = n_k^2 + \sum_{i=[k/2]}^{k}2n_i \left(3\frac{N}{n_i} -1\right) = 3Nk + o(N\sqrt{k}) \ \text{ as } k\to \infty,
\]
and hence multiplying by~$2$ gives the required expression.
We now turn to estimate the variance.
Using~\eqref{eq:decomp} we have
\begin{align*}
\vnk{\tau_0} &= \enk{\left( (S - \enk{S}) + (D - \enk{D})\right)^2 } 
\\
&= \vnk{S} + \vnk{D} + 2\enk{(S-\enk{S}) (D- \enk{D})}.
\end{align*}
Since $S$ is the first time that a simple random walk on $[0,n_k]\cap \Z$ hits $0$ started from $n_k$ it follows that 
\begin{align}\label{eq:vars}
\vnk{S} \asymp n_k^4 = o(N^2 k).
\end{align}
By Cauchy Schwarz we get
\[
\enk{(S-\enk{S}) (D- \enk{D})} \leq \sqrt{\vnk{S} \vnk{D}},
\]
so if we prove that 
\begin{align}\label{eq:goal}
\vnk{D} \asymp N^2 k,
\end{align}
then using~\eqref{eq:vars} we get $\sqrt{\vnk{S} \vnk{D}} \asymp N n_k^2 \sqrt{k} = o(N^2 k)$, and hence $\vnk{\tau_0} \asymp N^2 k$. Therefore, it suffices to show~\eqref{eq:goal}.

To simplify notation further we write $D_i = \sum_{\ell=1}^{L_i} T_\ell^{(i)}$.  We have
\begin{equation}
\begin{split}
\label{eq:variance}
\vnk{D} &= \sum_{i,j=[k/2]}^{k} \enk{(D_i - \enk{D_i})(D_j - \enk{D_j})}\\&= \sum_{j=[k/2]}^{k} \vnk{D_j} + 
2\sum_{j=[k/2]}^{k}\sum_{i=j+1}^{k} \enk{(D_i - \enk{D_i})(D_j - \enk{D_j})}.
\end{split}
\end{equation}
By Claims~\ref{cl:excursion} and~\ref{cl:localtime} and the independence between $L_i$ and $T_\ell^{(i)}$, we get that for all $i$
\[
\vnk{D_i} = \enk{T_1^{(i)}}^2 \vnk{L_i} + \enk{L_i}\vnk{T_1^{(i)}} \asymp N^2,
\]
and hence $\sum_{i=[k/2]}^{k} \vnk{D_i} \asymp N^2k$.
In view of that, it suffices to show that for~$i > j$
\begin{align}\label{eq:goalnow}
\left| \enk{(D_i - \enk{D_i})(D_j - \enk{D_j})}    \right| \lesssim N^2\frac{n_j}{n_i},
\end{align}
since then using the double exponential decay of $(n_\ell)$ completes the proof of the lemma.

Since in order to hit $0$ starting from $n_k$ the random walk must first hit $n_i$ and then $n_j$, it makes sense to split the local time $L_i$ into two terms: the time $L_{i,1}$ that $Y$ spends at $n_i$ before the first hitting time of $n_j$ and the time $L_{i,2}$ that $Y$ spends at $n_i$ after the first hitting time of $n_j$. Writing 
\[
D_{i,1} = \sum_{\ell=1}^{L_{i,1}} T_\ell^{(i)} \  \ \text{ and } \ \
D_{i,2} = \sum_{\ell=1}^{L_{i,2}} \til{T}_\ell^{(i)}, 
\]
where $\til{T}$ is an independent copy of $T$,
we have that~$D_{i,1}$ is independent of~$D_j$, and hence
\begin{align}
\label{eq:kkk}
\enk{(D_i- \enk{D_i})(D_j - \enk{D_j})} = \enk{(D_{i,2} - \enk{D_{i,2}})(D_j - \enk{D_j})}.
\end{align}
Using the independence between the local times and the delays we get
\begin{equation}
\begin{split}
\label{eq:ddd}
\enk{D_{i,2} D_j} = \enk{\econds{\sum_{\ell = 1}^{L_{i,2}}\til{T}_{\ell}^{(i)} \sum_{r=1}^{L_j} T_r^{(j)}}{L_{i,2}, L_j}{n_k}} &= \enk{L_{i,2}\enk{T_1^{(i)}} L_{j} \enk{T_1^{(j)}}} \\
&= \enk{L_{i,2} L_j} \enk{T_1^{(i)}} \enk{T_1^{(j)}}.
\end{split}
\end{equation}
If we denote by $\tau_x$ the hitting time of $x$ by the random walk $Y$, then we get
\[
\prstart{\tau_{n_i} < \tau_0 \wedge \tau_{n_j}^+}{n_j} = \frac{1}{2(n_i - n_j)}.
\]
Once the random walk $Y$ visits $n_i$, then the total number of returns to $n_i$ before hitting $n_j$ again is a geometric random variable independent of $L_j$ and of parameter 
\[
\prstart{\tau_{n_j}<\tau_{n_i}^+}{n_i}= \frac{1}{2(n_i-n_j)}.
\]
Hence we can write
\[
L_{i,2} = \sum_{\ell=1}^{L_j-1} \eta_\ell,
\]
where $\eta_\ell = 0$ with probability $1-1/(2(n_i - n_j))$ and $\theta_\ell$ with probability $1/(2(n_i - n_j))$, where $\theta_\ell$ is a geometric random variable with $\E{\theta_\ell} = 2(n_i - n_j)$. Note that $\eta_\ell$ is independent of $L_j$. Therefore we deduce
\begin{align*}
\enk{L_{i,2}L_j} = \enk{L_j\sum_{\ell=1}^{L_j-1}\eta_{\ell}} = \enk{\econd{L_j \sum_{\ell=1}^{L_j-1} \eta_\ell}{L_j}} = (\enk{L_j^2} - \enk{L_j}) \enk{\eta_1} \asymp n_j^2,
\end{align*}
where in the last step we used Claim~\ref{cl:localtime} and the fact that $\enk{\eta_\ell} =1$ for all $\ell$.
Hence combining the above with~\eqref{eq:ddd} and Claim~\ref{cl:excursion} we conclude
\begin{align}\label{eq:final}
\enk{D_{i,2}D_j} \asymp n_j^2 \frac{N}{n_i} \frac{N}{n_j} = N^2 \frac{n_j}{n_i}.
\end{align}
Using Wald's identity we obtain
\begin{align*}
\enk{D_{i,2}} \enk{D_j} \asymp N^2 \frac{n_j}{n_i}
\end{align*}
and combined with~\eqref{eq:kkk} and~\eqref{eq:final} proves~\eqref{eq:goalnow} and thus finishes the proof of the lemma.
\end{proof}

\begin{proof}[{\bf Proof of Theorem~\ref{thm:tree}} (lower bound)]

Let $t = \enk{\tau_0} - \gamma \sqrt{\vnk{\tau_0}}$, where the constant $\gamma$ will be determined later. 
By the definition of the total variation distance we get
\begin{align*}
d(t) \geq \| \prstart{X_t \in \cdot}{n_k} - \pi\| \geq \pi(\T_0) - \prstart{X_t \in \T_0}{n_k} \geq 1-o(1)- \prstart{\tau_0<t}{n_k},
\end{align*}
since $\pi(\T_0) = 1 - o(1)$.
Chebyshev's inequality gives
\begin{align*}
\prstart{\tau_0<t}{n_k} \leq \prstart{\left|\tau_0 -\enk{\tau_0}\right| > \gamma \sqrt{\vnk{\tau_0}}}{n_k} \leq \frac{1}{\gamma^2}.
\end{align*}
Hence by choosing $\gamma$ big enough we deduce that for all sufficiently large $k$ 
\begin{align*}
d(t) \geq 1 - o(1)- \frac{1}{\gamma^2} > \epsilon,
\end{align*}
which implies that $\tmix(\epsilon) \geq t$. By Lemma~\ref{lem:concentration} we thus get that
\[
\tmix(\epsilon) \geq 6Nk - c_1N\sqrt{k}
\]
for a positive constant $c_1$.
\end{proof}

\section{Coupling}\label{sec:coupling}

In this section we prove the upper bound on $\tmix(\epsilon)$ via coupling.

\begin{proof}[{\bf Proof of Theorem~\ref{thm:tree}} (upper bound)]

Let $X_0 = x$ and $Y_0\sim \pi$. Consider the following coupling. We let $X$ and $Y$ evolve independently until the first time that 
$X$ hits $0$. After that we let them continue independently until the first time they collide or reach the same level of the tree $\T_0$ in which case we change the coupling to the following one: we let $X$ evolve as a lazy simple random walk and couple $Y$ to $X$ so that $Y$ moves closer to (or further from) the root if and only if $X$ moves closer to (or further from) the root respectively. Hence they coalesce if they both hit $0$.

Let $\tau$ be the coupling time and $t = \enk{\tau_0} + \gamma\sqrt{\vnk{\tau_0}}$, where the constant $\gamma$ will be determined later in order to make $\pr{\tau>t}$ as small as we like.

Define $\tau_x^* = \inf\{ s\geq \tau_0: X_s = x \}$ for all $x$ and 
\begin{align*}
L = \sum_{s=\tau_0}^{\tau_{n_{[k/2]}}^*} \1\left(X_{s-1} \notin \T_0, X_s \in \T_0\right),
\end{align*}
i.e.\ $L$ is the number of returns to the tree $\T_0$ in the time interval $[\tau_0,\tau^*_{n_{[k/2]}}]$.
Then $L$ has the geometric distribution with parameter $1/n_{[k/2]}$. Setting $A_L=\{L> \sqrt{n_{[k/2]}}\}$
we get by the union bound 
\begin{align}\label{eq:unionbd}
\pr{A_L^c}=\pr{L\leq \sqrt{n_{[k/2]}}} \leq \frac{1}{\sqrt{n_{[k/2]}}}.
\end{align}

We also define the event that after time $\tau_0$ the random walk hits the leaves of the tree $\T_0$ before exiting the interval $[0,n_{[k/2]}]$, i.e.
\begin{align}\label{eq:defe}
E = \left\{ \tau^*_{\partial \T_0} < \tau^*_{n_{[k/2]}} \right\}.
\end{align}
Since at every return to the tree $\T_0$ with probability at least $1/3$ the random walk hits the leaves of~$\T_0$ before exiting the tree~$\T_0$, it follows that 
\begin{align}\label{eq:hittheleaves}
\prcond{\tau_{\partial \T_0}^* >\tau_{n_{[k/2]}}^*}{L}{} \leq \left(\frac{2}{3}\right)^{L}.
\end{align}
By decomposing into the events $A_L$ and $E$ we obtain
\begin{align}\label{eq:bigeq}
\nonumber\pr{\tau>t} &\leq \pr{\tau>t, A_L} + \frac{1}{\sqrt{n_{[k/2]}}} \\
& \leq \pr{\tau>t,A_L, E} +\left(\frac{2}{3}\right)^{\sqrt{n_{[k/2]}}}  + \frac{1}{\sqrt{n_{[k/2]}}}, 
\end{align}
where the first inequality follows from~\eqref{eq:unionbd} and the second one from~\eqref{eq:hittheleaves} and the fact that we are conditioning on the event $\{L>\sqrt{n_{[k/2]}}\}$.   
We now define $S$ to be the first time after $\tau^*_{\partial \T_0}$ that $X$ hits $0$, i.e.\ $S= \inf\{s\geq \tau^*_{\partial \T_0}: X_s =0\}$. 

Let~$(\xi_i)_i$ be i.i.d.\ random variables, where~$\xi_1$ is distributed as the length of a random walk excursion on the interval~$[0,n_{[k/2]}]$ conditioned not to hit~$n_{[k/2]}$. Let~$(\ell_{i,j})_{i,j}$ be i.i.d.\ random variables with~$\ell_{1,1}$ distributed as the length of a random walk excursion from the root on the tree~$\T_0$ conditioned not to hit the leaves and~$(G_i)_i$ be i.i.d.\ geometric random variables of success probability~$1/3$. Then on the event~$E$ we have
\[
S - \tau_0 \prec \sum_{i=1}^{L} \xi_i + \sum_{i=1}^{L} \sum_{j=1}^{G_i} \ell_{i,j} +\zeta,
\]
where~$\zeta$ is independent of the excursion lengths and is distributed as the commute time between the root and the leaves of the tree~$\T_0$ and~$\prec$ denotes stochastic domination. Hence, by Wald's identity we obtain
\begin{align}\label{eq:expect}
\E{(S-\tau_0)\1(E)} \leq \E{L} \E{\xi_1} + \E{L} \E{G_1} \E{\ell_{1,1}} + \E{\zeta} \lesssim n_{[k/2]}^2 + n_{[k/2]} + N \lesssim N.
\end{align} 
Let $A=\{Y_{\tau_0} \in \T_0\}$. Then $\pr{A^c} = o(1)$ as $k\to \infty$, because at time $\tau_0$ the random walk $Y$ is stationary, since until this time it evolves independently of $X$, and also the stationary probability of the tree is $1-o(1)$. It then follows
\begin{align}\label{eq:one}
\pr{\tau>t, A_L,E} \leq \pr{\tau>t, A_L, E,A} +o(1).
\end{align}
Let $\tau_1$ be the time it takes to hit the line $[0,n_k]$ starting from $x$. Let $\tau_2$ be the time it takes to hit $0$ starting from $X_{\tau_1}$. Then clearly $\tau_2$ is smaller than the time it takes to hit $0$ starting from $n_k$. Thus setting $B= \{\tau_0<\enk{\tau_0} + \gamma \sqrt{\vnk{\tau_0}}/2\}$ we obtain
\begin{align}\label{eq:two}
\nonumber\prstart{B^c}{x} &\leq \prstart{\tau_2>\enk{\tau_0} + \frac{\gamma \sqrt{\vnk{\tau_0}}}{4}}{x} +\prstart{\tau_1\geq\frac{\gamma\sqrt{\vnk{\tau_0}}}{4}}{x} \\
&\leq \prstart{\tau_0>\enk{\tau_0}+ \frac{\gamma\sqrt{\vnk{\tau_0}}}{4}}{n_k} + o(1)
\leq \frac{16}{\gamma^2} + o(1),
\end{align}
where the second inequality follows from Markov's inequality and the fact that $\estart{\tau_1}{x} \leq N$ for all $x$ and the third one follows from Chebyshev's inequality. Ignoring the $o(1)$ terms we get
\begin{align}\label{eq:three}
\pr{\tau>t,A_L,E,A} \leq \pr{\tau>t,A_L,E,A,B} + \frac{16}{\gamma^2}. 
\end{align}
We finally define the event $F=\{S-\tau_0 >\gamma\sqrt{\vnk{\tau_0}}/2\}$. We note that on the events $E$ and $A$ the two walks $X$ and $Y$ must have coalesced by time $S$. (Indeed, if $Y$ stays in $\T_0$ during the time interval $[\tau_0,\tau^*_{\partial \T_0}]$, then they must have coalesced. If $Y$ leaves the interval, since $X$ is always in $[0,n_{[k/2]}]$ until time $S$ on the event $E$, then coalescence must have happened again.) Therefore
\[
B\cap F^c \subseteq \{\tau<t\}.
\]
This in turn implies that for a positive constant $c_1$
\begin{align}\label{eq:five}
\pr{\tau>t,A_L,E,A,B} = \pr{\tau>t,A_L,E,A,B,F} \leq \pr{E,F} \leq \frac{c_1}{\gamma\sqrt{k}},
\end{align}
where the last inequality follows by applying Markov's inequality to $(S-\tau_0)\1(E)$ and using~\eqref{eq:expect} and Lemma~\ref{lem:concentration}.
Plugging~\eqref{eq:one},~\eqref{eq:two},~\eqref{eq:three} and~\eqref{eq:five} into~\eqref{eq:bigeq} gives as $k\to \infty$
\begin{align*}
\pr{\tau>t} \leq \frac{16}{\gamma^2} + o(1).
\end{align*}
Hence choosing $\gamma$ sufficiently large depending on $\epsilon$ we can make $\pr{\tau>t}<\epsilon$ and this shows that for a positive constant $c_2$
\[
\tmix(\epsilon) \leq \enk{\tau_0} + \gamma_\epsilon \sqrt{\vnk{\tau_0}} \leq 6Nk + c_2N\sqrt{k},
\]
where the last inequality follows by Lemma~\ref{lem:concentration}.
Combining this with the lower bound on $\tmix(\epsilon)$ proved in the previous section shows that there exists $c_\epsilon>0$ such that for all $0<\epsilon<1$ 
\[
\tmix(\epsilon) - \tmix(1-\epsilon) < c_\epsilon N\sqrt{k}
\]
and this completes the proof of the theorem.
\end{proof}

\section{Relaxation time}\label{sec:relaxation}

In this section we give the proof of Theorem~\ref{thm:relax}. We start by stating standard results for random walks on the interval $[0,n]$ and the binary tree. We include their proofs here for the sake of completeness. 
A detailed analysis of relaxation time for birth and death chains can be found in~\cite{ChenSaloff}.

\begin{claim}\label{cl:srwline}
Let $f$ be a function defined on $[0,n]$ satisfying $f(0)= 0$. Then 
\[
\sum_{k=1}^{n} f(k)^2 \leq n^2 \sum_{\ell=1}^{n}(f(\ell) - f(\ell-1))^2.
\]
\end{claim}

\begin{proof}[{\bf Proof}]

We set $\beta_\ell^{-2}= (n-\ell)$ for all $\ell \in [0,n]$. Then by Cauchy Schwarz we get
\begin{align*}
f(k)^2 = \left(\sum_{\ell=1}^{k}(f(\ell) - f(\ell-1)) \right)^2  \leq \sum_{\ell=1}^{k} \beta_\ell^2(f(\ell) - f(\ell-1))^2  \sum_{\ell=1}^{k} \beta_\ell^{-2}.
\end{align*}
Since $\sum_{\ell=1}^{k} \beta_\ell^{-2} \leq n^2$ for all $k \in [0,n]$ we get summing over all $k$ and interchanging sums 
\begin{align*}
\sum_{k=1}^{n} f(k)^2 \leq n^2 \sum_{\ell=1}^{n} (f(\ell) - f(\ell-1))^2
\end{align*}
and this completes the proof of the claim.
\end{proof}

\begin{claim}\label{cl:binarytree}
Let $\T$ be a binary tree on $m$ vertices with root $o$. Then there exists 
a universal constant~$c$ such that for all functions~$g$ defined on~$\T$ with $g(o)=0$ we have
\[
\|g\|^2 \leq cm\EE(g,g),
\]
where $\|g\|^2 = \sum_{x} \pi(x) g(x)^2$ and $\EE$ is the Dirichlet form $\EE(f,g) = \langle f,(I-P)g\rangle$. Here $\pi$ and $P$ are the stationary distribution and transition matrix of a simple random walk on $\T$ respectively.
\end{claim}

\begin{proof}[{\bf Proof}]
Since the stationary measure of a simple random walk on the tree satisfies $\pi(x) \asymp m^{-1}$ for all $x$ and $P(x,y) \asymp c$ for all $x\sim y$, we will omit them from the expressions. 

Let the depth of the tree $\T$ be $n = \lceil\log_2 m\rceil$.
Let $x_k$ be a vertex in $\T$ of level $k$. Then there exists a unique path $x_0=o, x_1, \ldots, x_k$ going from the root to $x_k$.  We can now write
\begin{align*}
g(x_k)^2 &= \left( \sum_{j=1}^{k} (g(x_{j-1})-g(x_j))\right)^2 = \left( \sum_{j=1}^{k} (g(x_{j-1})-g(x_j))2^{j/2} \frac{1}{2^{j/2}}\right)^2 \\
&\lesssim  \sum_{j=1}^{k} 2^j(g(x_{j-1})-g(x_j))^2, 
\end{align*}
where the last inequality follows by Cauchy Schwarz. 
Let $L_k$ denote all the vertices of the tree at distance $k$ from the root. For any $x \in \T$ we write 
\[
G(x) = \sum_{j=1}^{|x|} 2^j(g(y_j) - g(y_{j-1}))^2,
\]
where $|x|$ denotes the level of $x$ and $y_0=o, y_1,\ldots, y_{|x|}=x$ is the unique path joining $x$ to the root. 
By interchanging sums we obtain
\begin{align}\label{eq:bigsum}
\sum_{x \in \T_0} g(x)^2 = \sum_{k=1}^{n} \sum_{x\in L_k} g(x)^2 \lesssim \sum_k \sum_{x\in L_k} G(x).
\end{align}
Let $e$ be an edge of $\T$. We write $e=\langle e^-,e^{+}\rangle$ where $d(e^-,0)<d(e^+,0)$. For every edge $e$ we let $N(e)$ be the number of times the term $(g(e^-) - g(e^+))^2$ appears in the sum appearing on the right hand side of~\eqref{eq:bigsum}. We then get
\begin{align*}
\sum_{x \in \T} g(x)^2 \lesssim \sum_{e\in \T} 2^{|e^-|}N(e)(g(e^-) -g(e^+))^2. 
\end{align*}
Notice that $N(e)$ is the number of paths in $\T$ joining the root to the leaves and pass through $e$. Hence since the tree is of depth 
$n$ we get that $N(e) = 2^{n-|e^-|-1}$. Therefore we deduce
\begin{align*}
\sum_{x \in \T} g(x)^2  \lesssim \sum_{e\in \T} 2^{|e^-|} 2^{n-|e^-| -1} (g(e^-) - g(e^+))^2 = 2^n\sum_{e\in \T} (g(e^-) - g(e^+))^2 = m \EE(g,g)
\end{align*}
and this completes the proof of the claim. 
\end{proof}

\begin{proof}[{\bf Proof of Theorem~\ref{thm:relax}}]

To prove the lower bound on $\trel$ we use the bottleneck ratio as in~\cite[Theorem~13.14]{LevPerWil}. By setting $S = \T_0$, we see that
\[
\Phi_* \lesssim \frac{1}{N},
\]
and hence $\trel \gtrsim N$. It remains to prove a matching upper bound. 
We do that by using the variational formula for the spectral gap, which gives 
\begin{align*}
\trel = \sup_{\vars{f}{\pi} \neq 0} \frac{\vars{f}{\pi}}{\EE(f,f)}.
\end{align*}
Notice that by subtracting from $f$ its value at $0$ the ratio above remains unchanged. So we restrict to functions $f$ with $f(0) = 0$. It suffices to show that for any such $f$ 
\begin{align}\label{eq:relgoal}
\vars{f}{\pi} \lesssim N\sqrt{k} \EE(f,f).
\end{align}
Let $f$ be defined on the tree $\T$ with $f(0) = 0$. Then we can write $f= g + h$, where $g$ is zero on $\T_0^c$ and $h$ is zero on $\T_0$ and $g(0) = h(0) = 0$. We then have
\begin{align}\label{eq:new}
\vars{f}{\pi} \leq \|g+ h\|^2 = \|g\|^2 + \|h\|^2,
\end{align}
since by the definition of the functions $g$ and $h$ it follows that $\langle g,h\rangle_{\pi} =0$. 
Similarly we also get
\[
\EE(f,f) = \EE(g,g) + \EE(h,h).
\]

\begin{claim}\label{cl:restoftree}
There exists a positive constant $c$ such that
\[
\|h\|^2 \leq cN \EE(h,h).
\]
\end{claim}

\begin{proof}[{\bf Proof}]

Using Claim~\ref{cl:binarytree} for the function $(h(x) - h(n_j))$ restricted to $x \in \T_j$ we obtain
\begin{align}\label{eq:restrj}
\sum_{v \in \T_j} (h(v) - h(n_j))^2 \lesssim \frac{N}{n_j} \sum_{\substack{u,v \in \T_j \\ u\sim v}} (h(u) - h(v))^2.
\end{align}
Using that $(a+b)^2\leq 2a^2 + 2b^2$ and~\eqref{eq:restrj} we get
\begin{align*}
\sum_{v\in \T_j} h(v)^2 \leq 2\sum_{v\in \T_j} (h(v) - h(n_j))^2 + 2 \frac{N}{n_j} h(n_j)^2 \lesssim \frac{N}{n_j} \left(\EE(h,h) + h(n_j)^2\right).
\end{align*}
From the above inequality it immediately follows
\begin{align*}
\sum_{v\notin \T_0} h(v)^2 \leq \sum_{v \in [0,n_k]} h(v)^2 + \sum_{j=[k/2]}^{k} \sum_{v\in \T_j} h(v)^2 \lesssim \sum_{v\in [0,n_k]} h(v)^2 + \sum_{j=[k/2]}^{k} \frac{N}{n_j} h(n_j)^2 + N\EE(h,h)
\end{align*}
and hence it suffices to show
\begin{align}\label{eq:goalh}
\sum_{v\in [0,n_k]} h(v)^2 + \sum_{j=[k/2]}^{k} \frac{N}{n_j} h(n_j)^2 \lesssim N\sum_{\ell=1}^{n_k} (h(\ell) - h(\ell-1))^2.
\end{align}
Claim~\ref{cl:srwline} gives 
\begin{align*}
\sum_{v\in [0,n_k]} h(v)^2 \leq n_k^2 \sum_{\ell=1}^{n_k} (h(\ell) - h(\ell-1))^2 \leq N \sum_{\ell=1}^{n_k} (h(\ell) - h(\ell-1))^2,
\end{align*}
since $N=n_k^3$, and hence it suffices to show 
\begin{align}\label{eq:finalh}
 \sum_{j=[k/2]}^{k} \frac{h(n_j)^2 }{n_j} \lesssim \sum_{\ell=1}^{n_k} (h(\ell) - h(\ell-1))^2.
\end{align}
Setting $\Delta_\ell = h(\ell) - h(\ell-1)$ and using Cauchy Schwarz
\begin{align*}
h(n_j)^2 &\leq 2h(n_{j-1})^2 + 2(h(n_j) - h(n_{j-1}))^2 = 2 \left(\sum_{\ell=1}^{n_{j-1}} \Delta_\ell\right)^2 + 2\left(\sum_{\ell=n_{j-1}+1}^{n_j} \Delta_\ell\right)^2 \\
&\leq 2n_{j-1} \sum_{\ell=1}^{n_{j-1}} \Delta_\ell^2 + (n_j-n_{j-1})		\sum_{\ell=n_{j-1}+1}^{n_j} \Delta_\ell^2,	
\end{align*} 
and hence dividing by $n_j$ we get
\begin{align*}
\sum_{j=[k/2]}^{k} \frac{h(n_j)^2}{n_j} \leq 2\sum_{j=[k/2]}^{k} \frac{n_{j-1}}{n_j} \sum_{\ell=1}^{n_{j-1}} \Delta_\ell^2 + 2 \sum_{\ell=1}^{n_k} \Delta_\ell^2.
\end{align*}
If we fix $\ell \in [0,n_k]$, then the coefficient of $\Delta_\ell^2$ in the first sum appearing on the right hand side of the above inequality is bounded from above by $\sum_{j=1}^{k} n_{j-1}/n_j <\infty$, and hence we conclude
\[
\sum_{j=[k/2]}^{k} \frac{h(n_j)^2}{n_j}\lesssim \sum_{\ell=1}^{n_k} \Delta_\ell^2
\]
and this finishes the proof of the claim.
\end{proof}
Since $g$ satisfies the assumptions of Claim~\ref{cl:binarytree} it follows that
\[
\|g\|^2\leq cN\EE(g,g).
\]
This together with Claim~\ref{cl:restoftree} and~\eqref{eq:new} proves~\eqref{eq:relgoal} and completes the proof of the theorem.
\end{proof}

\bibliographystyle{plain}
\bibliography{biblio}

\end{document}